\documentclass[reqno]{amsart}
\RequirePackage[utf8]{inputenc}
\usepackage{amsmath,amsthm,amsfonts,amssymb,amscd,amsbsy,multirow,hyperref}

\usepackage{enumerate}

\newtheorem*{acknowledgement}{Acknowledgement}

\newtheorem{proposition}{Proposition}
\newtheorem{remark}{Remark}
\newtheorem{theorem}{Theorem}
\newtheorem{example}{Example}

\numberwithin{equation}{section}

\begin{document}

\title[Quasi-Einstein manifolds]{Noncompact quasi-Einstein manifolds\\ conformal to a Euclidean space}
\author{E. Ribeiro Jr.} 
\author{K. Tenenblat}

\address[E. Ribeiro Jr.]{Universidade Federal do Cear\'a - UFC, Departamento  de Matem\'atica, Campus do Pici,  Av. Humberto Monte, Bloco 914, 60455-760, Fortaleza - CE, Brazil}\email{ernani@mat.ufc.br}

\address[K. Tenenblat]{Universidade de Bras\'ilia - UnB, Departamento de Matem\'atica, 70910-900, Bras\'ilia - DF, Brazil} 
\email{k.tenenblat@mat.unb.br}

\thanks{E. Ribeiro was partially supported by grants from CNPq/Brazil [Grant: 305410/2018-0], PRONEX - FUNCAP /CNPq/ Brazil and CAPES/ Brazil - Finance Code 001.}

\thanks{K. Tenenblat was partially supported by CNPq/Brazil  [Grant: 312462/2014-0], CAPES/ Brazil - Finance Code 001 and FAPDF/Brazil [Grant: 0193.001346/2016]}



\newcommand{\spacing}[1]{\renewcommand{\baselinestretch}{#1}\large\normalsize}
\spacing{1.2}

\begin{abstract}
The goal of this article is to investigate nontrivial $m$-quasi-Einstein ma\-ni\-folds globally con\-formal to an $n$-dimensional Euclidean space. By considering such manifolds, whose conformal factors and potential functions are  invariant under the action of an $(n-1)$-dimensional translation group, we provide a complete classification when $\lambda=0$ and $m\geq 1$ or $m=2-n.$
\end{abstract}

\maketitle

\section{Introduction}

A distinguished problem in Riemannian geometry is to find canonical metrics on a given manifold. For example, it is common to look for Einstein metrics on a given smooth manifold. Einstein and Hilbert proved that the critical points of the total scalar curvature functional, restricted to the set of smooth Riemannian structures on a compact manifold $M^n$ of unitary volume, must be ne\-cessarily Einstein (see \cite[Theorem 4.21]{Besse}), and this suggests that Einstein metrics are in fact special. They are not only interesting in themselves but they are also related to many important topics of Riemannian geometry. In this scenario, it is very important to build new explicit examples of Einstein metrics. As discussed by Besse \cite[pg. 265]{Besse}, one promising way to construct Einstein metrics is by imposing symmetry, such as by considering warped products. It is known that the $m$-Bakry-Emery Ricci tensor, which appeared
previously in \cite{bakry,Besse} and \cite{Qian}, is useful as an attempt to better understand Einstein warped products. More precisely, the $m$-Bakry-Emery Ricci tensor is given by
\begin{equation}
\label{bertens}
Ric_{f}^{m}=Ric+\nabla ^2f-\frac{1}{m}df\otimes df,
\end{equation} where $f$ is a smooth function on $M^n$ and $\nabla ^2f$ stands for the Hessian of $f.$ We remark that it is also used to study the weighted measure $d\mu=e^{-f}dx,$ where $dx$ is the Riemann-Lebesgue measure determined by the metric. 

According to \cite{CaseShuWey}, a complete Riemannian manifold $(M^n,\,g),$ $n\geq 2,$ will be called $m$-{\it quasi-Einstein manifold}, or simply {\it quasi-Einstein manifold}, if there exists a smooth potential function $f$ on $M^n$ satisfying the following fundamental equation
\begin{equation}
\label{eqqem}
Ric_{f}^{m}=Ric+\nabla ^2f-\frac{1}{m}df\otimes
df=\lambda g,
\end{equation} for some constants $\lambda$ and $m\neq 0.$ It is also important to recall that, on a quasi-Einstein manifold, there is an indispensable constant  $\mu$ such that 
\begin{equation}\label{2eq}
\Delta f-|\nabla f|^{2} = m\lambda-m\mu e^{\frac{2}{m}f}.
\end{equation} For more details, we refer the reader to \cite{Kim}.

We say that a quasi-Einstein manifold is \emph{trivial} if its potential function $f$ is constant, otherwise, we say that it is \emph{nontrivial}. Hence, the triviality implies that $M^n$ is an Einstein manifold. An $\infty$-quasi-Einstein manifold is a gradient Ricci soliton. Ricci solitons model the formation of singularities in the Ricci flow and correspond to self-similar solutions, i.e., solutions which evolve along symmetries of the flow, see \cite{Cao} and references therein for more details on this subject. We also remark that $1$-quasi-Einstein manifolds are more commonly called {\it static metrics} and such metrics have connections to the prescribed scalar curvature pro\-blem, the positive mass theorem and general relativity. On the other hand, when $m$ is a positive integer it corresponds to a warped product Einstein metric (see \cite{Besse,CaseShuWey}). Indeed, a motivation to study quasi-Einstein metrics on a Riemannian manifold is its direct relation to the existence of Einstein warped products, which also have different pro\-perties compared with the gradient Ricci solitons; for more details see, for instance, Theorem 1 in \cite{Ernani2} or Corollary 9.107 in \cite[pg. 267]{Besse}. Another important motivation comes from the study of diffusion operators by Bakry and \'Emery \cite{BE}.

In \cite{BRS,Besse} and \cite{Wang1}, the authors gave some examples of complete $m$-quasi-Einstein manifolds with $\lambda<0$ and arbitrary $\mu,$ as well as examples of quasi-Einstein manifolds with $\lambda=0$ and $\mu>0.$ Case \cite{Case} showed that complete $m$-quasi-Einstein manifolds with $\lambda=0$ and $\mu\leq0$ are trivial. While Qian \cite{Qian} proved that complete $m$-quasi-Einstein manifolds with $\lambda>0$ must be compact. Moreover, by Kim and Kim \cite{Kim} nontrivial compact quasi-Einstein manifolds must have $\lambda>0.$ Thereby, it follows that a complete  nontrivial quasi-Einstein manifold is compact if, and only if, $\lambda>0$ (see also \cite[Theorem 4.1]{HPW}). An example of nontrivial compact $m$-quasi-Einstein manifold with $\lambda>0,$ $m>1$ and $\mu>0$ was obtained in \cite{LuePage}. Other complete examples were obtained by He, Petersen and Wylie \cite{HPW} on the hyperbolic space. An alternative description of the known examples on hyperbolic space was given by Case \cite{CaseT} using tractors; see also \cite{rimoldi,rimoldi2,Wang2} for further related results. 

In this paper, we will consider nontrivial $m$-quasi Einstein manifolds (not necessarily complete) with $\lambda\leq 0,$ which are globally conformal to an $n$-dimensional Euclidean space, whose conformal factors and potential functions  are invariant  under the action of an $(n-1)$-dimensional translation group.   Solutions of geometric PDEs, which are invariant under the action  of such a group,  were  obtained in \cite{BPK}, where  Barbosa, Pina and Tenenblat studied such solutions for  gradient Ricci solitons conformal to an $n$-dimensional pseudo-Euclidean space.  In particular, they classified all such gradient Ricci solitons in the steady case (i.e. $\lambda=0$). Later, similar kind of solutions were obtained for gradient Yamabe solitons in \cite{KB}; for the Ricci curvature equation and the Einstein field equation in \cite{PS} and \cite{BLP}.

Now we may state our main results. The first one provides a uniqueness result for noncompact, nontrivial $m$-quasi-Einstein manifolds, with $\lambda=0$ and $m+n-2=0,$ that are conformal to a Euclidean space. More precisely, we have established the following result.

\begin{theorem}
\label{thmnovo}
Let  $(\mathbb{R}^{n},\,g),$ $n\ge 3,$ be a Euclidean space with coordinates $x=(x_{1},...,x_{n})$ and $g_{ij}=\delta_{ij}.$ Consider smooth functions 
$\varphi(\xi)$ and $u(\xi)>0$,  $\xi=\sum_{i=1}^{n}\alpha_{i}x_{i}$,  
$\alpha_{i}\in \mathbb{R},$ where without loss of generality we consider $\sum_{i=1}^n\alpha_i^2=1$.  Then $\overline{g}=\frac{1}{\varphi^{2}}g$   
is a nontrivial $m$-quasi-Einstein metric with potential function 
$f=-m \log u,$ $\lambda=0$, $\varphi$ non-constant and $m+n-2=0$ if, and only if, $\varphi$ and $u$ are given by 
\begin{equation}
u=C_2|\xi+C_1|e^{-C_3(\xi+C_1)^{n-1}}
\qquad\hbox{and} \qquad \varphi=\pm C_4 e^{C_3(\xi+C_1)^{n-1}}, 
\end{equation}
where $C_1\in\mathbb{R}$ and $C_2,\,C_3,\, C_4$ are positive real numbers. Moreover, the sign of $\varphi$  is the sign of $\xi+C_1\neq 0$ and the potential function  $f$ is given by  
\[
f=(n-2)\Big[\log \big(C_2|\xi+C_1|\big)-C_3\big(\xi+C_1\big)^{n-1}\Big], 
\]
which is defined on $\mathbb{R}^n\setminus \Pi,$ where $\Pi$ is the hyperplane $\xi+C_1= 0$. 
\end{theorem}

In our next result, we characterize the noncompact $m$-quasi-Einstein manifolds with $\lambda=0$ and $m\ge 1$ that are conformal to a Euclidean space. To be precise, we have the following result.

\begin{theorem}
\label{thmB1}
Let  $(\mathbb{R}^{n},\,g),$ $n\ge 3,$ be a Euclidean space with coordinates $x=(x_{1},...,x_{n})$ and $g_{ij}=\delta_{ij}.$ Consider smooth functions $\varphi(\xi)$ and $u(\xi)$, where $\xi=\sum_{i=1}^{n}\alpha_{i}x_{i},$ $\alpha_{i}\in \mathbb{R}$ and 
$\sum_{i=1}^n \alpha_i^2=1$. Then $(\mathbb{R}^{n},\,\overline{g}=\frac{1}{\varphi^{2}}g,\,f)$ is a nontrivial $m$-quasi-Einstein manifold with $\varphi$ nonconstant, $m\ge 1$, $\lambda=0$ and $f=-m\log u$ as potential function if, and only if, $u$ is determined in terms of $\varphi$ by
\begin{equation}
\label{uvarphi}
u=C \varphi^{\frac{n-1}{m}} (\varphi')^{-\frac{1}{m}},
\end{equation} where $C$ is a positive constant. Moreover, $\varphi(\xi)$ is given implicitly as follows:

\begin{enumerate}
\item[i)] If $m=1,$ then 
\begin{equation}\label{intvarphi1}
\int{\frac{exp(\frac{C_1}{2\varphi^{2(n-1)} })}   {\varphi^{\frac{n}{2}}} }\,d\varphi=C_2\,\xi+C_3, 
\end{equation} 
where $C_{2}\neq 0$, $C_1$ and $C_{3}$ are constants.
\item[ii)]  If $m> 1,$ then $\varphi$ is implicitly given by    
\begin{equation}\label{intvarphi2}
\int{ \frac{d\varphi}{\left( C_1\varphi^{\sqrt{b}}-1\right)^{\frac{m}{m-1}}\varphi^{\frac{a}{2}}}}=C_2\xi+C_3, 
\end{equation}  where $C_1\neq 0$, $C_2\neq 0$ and $C_3$ are constants. Additionally, 
$$a=\frac{-2m}{m-1} \left[(m-1)+\frac{(n-1)}{m}+\frac{\sqrt{b}}{2}\right]$$ 
and $b$ is a positive constant given by 
$$ b= 4\left[(m-1)^2+\frac{n-1}{m}(3m+n-4) \right].$$ 
\end{enumerate}
\end{theorem}

We highlight that in both theorems above we have considered metrics $\bar{g}$ non-homothetic to the Euclidean metric $g.$ Indeed, the homothetic case occurs for any $m\neq 0,$ when $\lambda=0,$ the function $u$ is linear on $\xi$ and $f$ is defined on a half space. More precisely, in the homothetic case we immediately have the following observation.

\begin{remark} \label{rm1}  Consider $(\mathbb{R}^n,\bar{g}),$ where $\bar{g}=g/\varphi$ is 
homothetic to the Euclidean metric $g,$ i.e.,
  $\varphi=\gamma \in \mathbb{R}\setminus \{0\}$ and $u(\xi)>0,$ where $\xi=\sum_{i=1}^{n}\alpha_{i}x_{i}$,  $\alpha_{i}\in \mathbb{R}$ and without loss of generality we consider $\sum_{i=1}^n\alpha_i^2=1.$  Then $\bar{g}=\frac{g}{\gamma^2}$ is a nontrivial $m$-quasi-Einstein metric with potential function $f=-m\log u,$ $m\neq 0$ and $\lambda\leq 0$ if, and only if, $\lambda=0$ and $u(\xi)=a\xi+b,$ where $a\neq 0,$ $b$ are real numbers and $f$ is defined on the half space $a\xi+b>0$.    
\end{remark}

The proofs of Theorems \ref{thmnovo} and \ref{thmB1} will be presented in Section \ref{sec3}. We emphasize that the nontrivial quasi-Einstein metrics exhibited in this article are different from the previously known examples obtained by He-Petersen-Wylie \cite{HPW} (see also \cite{CaseT} and \cite{Besse}). Furthermore, it is important to highlight that Theorem \ref{thmB1} provides all quasi-Einstein manifolds conformal to a Euclidean space, with $\lambda=0$  and $m\ge 1,$ whose  nonconstant conformal factors and potential functions are invariant under the action of an $(n-1)$-dimensional translation group. In particular, by choosing $C_{1}=0$ in the first item of Theorem \ref{thmB1}, we obtain the following explicit example. 
 
 \vspace{.1in} 

\begin{example} Consider $(\mathbb{R}^n, \bar{g}),$ where the metric $\bar{g}$  is  conformal to the Euclidean metric $g$ given by   
\begin{equation*}
\bar{g}= \left(\frac{(n-2)^2(C_2\xi+C_3)^2}{4}\right)^\frac{2}{n-2} g.
\end{equation*}
Let 
\begin{equation*}
f=-\log\left(\frac{-2C}{C_2(n-2)(C_2\xi+C_3)} \right), 
\end{equation*} where $C>0$, $C_2\neq 0$ and $C_3$ are real numbers. Then the half space where $-C_2(C_2\xi+C_3)>0$ is a $1$-quasi-Einstein metric with potential function $f$ and $\lambda=0$. This example is obtained by considering $C_1=0$ in \eqref{intvarphi1} and by integrating the left hand side.
\end{example}

We point out that $C_1\neq 0$ in \eqref{intvarphi2}. Besides, it is not clear whether one can obtain simple solutions in case $ii)$ of Theorem \ref{thmB1}. For instance, by choosing $n=4$ and $m=5$ the integration of the left hand side provides hypergeometric functions.

\section{Preliminaries}

In this section,  we review some basic facts and we prove a couple of propositions that  will be useful in the proof of the main results. First of all, assuming that $m<\infty,$ we may consider the function $u = e^{-\frac{f}{m}}$ on $M^n.$ Hence, we immediately get $$\nabla u = -\frac{u}{m}\nabla f$$ as well as

\begin{equation}\label{fu}
Hess f - \frac{1}{m}df \otimes df = -\frac{m}{u}Hess\, u.
\end{equation} 
In particular, notice that (\ref{eqqem}) and (\ref{fu}) yield 

\begin{equation}\label{quasiforu}
Ric -\frac{m}{u}Hess\, u=\lambda g. 
\end{equation} 
Moreover, taking into account (\ref{eqqem}) and (\ref{2eq}), it is not difficult to show that
\begin{equation}\label{fundamentalequationforu}
\frac{u^2}{m}(R-\lambda n) + (m-1)|\nabla u|^2 = -\lambda u^2 + \mu,
\end{equation} where $R$ is the scalar curvature of $M^n.$

In the sequel we discuss two key results that will play a crucial role in the proofs of the main theorems. The first one provides the relation between the potential function of an $m$-quasi Einstein manifold conformal to the Euclidean space and its associated conformal factor.

\begin{proposition}
\label{thm1}
Let  $(\mathbb{R}^{n},\,g),$ $n\ge 3,$  be a Euclidean space with coordinates $x=(x_{1},...,x_{n})$ and $g_{ij}=\delta_{ij}.$ Consider a smooth function $u:\mathbb{R}^{n}\to \mathbb{R}$, $u>0$. Then, there exists a metric $\overline{g}=\frac{1}{\varphi^{2}}g$ such that $(\mathbb{R}^{n},\,\overline{g})$ is a nontrivial $m$-quasi-Einstein manifold with $f=-m\log u$ as a potential function if, and only if, the functions $\varphi$ and $u$ are related as follows

\[
\frac{m}{u}\big(u_{x_{i}x_{j}}+\frac{\varphi_{x_{j}}}{\varphi}u_{x_{i}}+\frac{\varphi_{x_{i}}}{\varphi}u_{x_{j}}\big)= (n-2)\frac{1}{\varphi}\varphi_{x_ix_j},\qquad \hbox{for}\,\,\,i\neq j,
\] 
and for all $i$
\[
\frac{m}{u}\big(u_{x_{i}x_{i}}+2\frac{\varphi_{x_{i}}}{\varphi}u_{x_{i}}-\sum_{k}\frac{\varphi_{x_{k}}}{\varphi}u_{x_{k}}\big)=(n-2)\frac{1}{\varphi}\varphi_{x_ix_i}+\frac{\Delta_g\varphi}{\varphi}-(n-1)\frac{|\nabla_{g}\varphi|^{2}}{\varphi^{2}}-\frac{\lambda}{\varphi^{2}}.
\]

\end{proposition}  

This result was previously obtained by Case \cite[Proposition 4.13]{CaseIJM} by using a different approach.  For the sake of completeness we include here an alternative detailed proof. 

\subsubsection{Proof of Proposition \ref{thm1}}
\begin{proof} The first part of the proof will follow the trend of \cite{BPK}. Indeed, taking into account that $\overline{g}=\frac{1}{\varphi^{2}}g,$ where $g$ is the Euclidean metric,  we have 

\begin{eqnarray}
\label{e23}
Ric_{\overline{g}}=\frac{1}{\varphi^{2}}\{ (n-2)\varphi Hess_{g}\varphi +[\varphi \Delta_{g}\varphi - (n-1)|\nabla_{g}\varphi|^{2}]g\}.
\end{eqnarray} 
For more details see, for instance, \cite{Besse}. Hence, we may use (\ref{e23}) to rewrite the fundamental equation (\ref{quasiforu}) with respect to $\overline{g}$ as follows

\begin{equation}
\label{eq345}
\frac{\lambda}{\varphi^{2}}\delta_{ij}+\frac{m}{u}\left(Hess_{\overline{g}}\,u\right)_{ij} =\frac{1}{\varphi^{2}}\{ (n-2)\varphi\, \varphi_{x_ix_j} +[\varphi \Delta_{g}\varphi - (n-1)|\nabla_{g}\varphi|^{2}]g_{ij}\}.
\end{equation} 

On the other hand, we recall that 

\begin{eqnarray}
\label{a13}
\big(Hess_{\overline{g}}\,u\big)_{ij}=u_{x_{i}x_{j}}-\sum_{k}\overline{\Gamma}_{ij}^{k}f_{x_{k}},
\end{eqnarray} 
where $\overline{\Gamma}_{ij}^{k}$ are the Christoffel symbols with respect to $\overline{g}.$ We also recall that, for $i,$ $j$ and $k$ distinct, we have
\begin{eqnarray}
\label{a14}
\left\{\begin {array}{cc}
	\overline{\Gamma}_{ij}^{k}=0\hspace{1.0cm}, &\ \overline{\Gamma}_{ij}^{i}=-\frac{\varphi_{x_{j}}}{\varphi}, \\
	\overline{\Gamma}_{ii}^{k}=\frac{\varphi_{x_{k}}}{\varphi}\,\,\,\,\hbox{and}&\ \overline{\Gamma}_{ii}^{i}=-\frac{\varphi_{x_{i}}}{\varphi}. \end {array}\right.
\end{eqnarray} 
Therefore, combining (\ref{a13}) and (\ref{a14}), with $i\neq j,$ we deduce

\begin{equation}
\label{ag1}
\big(Hess_{\overline{g}}\,u\big)_{ij}=u_{x_{i}x_{j}}+\frac{\varphi_{x_{j}}}{\varphi}u_{x_{i}}+\frac{\varphi_{x_{i}}}{\varphi}u_{x_{j}}.
\end{equation} 
Moreover, by considering $i=j$, from (\ref{a13}) and (\ref{a14}), we immediately have that

\begin{equation}
\label{ag2}
\big(Hess_{\overline{g}}\,u\big)_{ii}=u_{x_{i}x_{i}}+2\frac{\varphi_{x_{i}}}{\varphi}u_{x_{i}}-\sum_{k}\frac{\varphi_{x_{k}}}{\varphi}u_{x_{k}}.
\end{equation} 
Next, it suffices to substitute (\ref{ag1}) into (\ref{eq345}), for $i\neq j,$  to obtain that
\[
\frac{m}{u}\big(u_{x_{i}x_{j}}+\frac{\varphi_{x_{j}}}{\varphi}u_{x_{i}}+\frac{\varphi_{x_{i}}}{\varphi}u_{x_{j}}\big)= (n-2)\frac{\varphi_{x_ix_j}}{\varphi}.
\] 
Similarly, substituting (\ref{ag2}) into (\ref{eq345}) we get 

\begin{eqnarray*}
\frac{m}{u}\big(u_{x_{i}x_{i}}+2\frac{\varphi_{x_{i}}}{\varphi}u_{x_{i}}-\sum_{k}\frac{\varphi_{x_{k}}}{\varphi}u_{x_{k}}\big)=(n-2)\frac{\varphi_{x_ix_i}}{\varphi}
+\frac{\Delta_g\varphi}{\varphi}-(n-1)\frac{|\nabla_{g}\varphi|^{2}}{\varphi^{2}}-\frac{\lambda}{\varphi^{2}}. 
\end{eqnarray*} This concludes the proof of Proposition \ref{thm1}.

\end{proof}

Our next result characterizes the quasi-Einstein manifolds conformal to the Euclidean space, whenever the conformal factor and the potential functions are  invariant under the action of an $(n-1)$-dimensional translation group. 

\begin{proposition}
\label{lemA}
Let  $(\mathbb{R}^{n},\,g),$ $n\ge 3,$ be a Euclidean space with coordinates $x=(x_{1},...,x_{n})$ and $g_{ij}=\delta_{ij}.$ Consider smooth functions $\varphi(\xi)$ and $u(\xi)$,  $\xi=\sum_{i=1}^{n}\alpha_{i}x_{i},$ $\alpha_{i}\in \mathbb{R}$, where without loss of generality we consider $\sum_{i=1}^n\alpha_i^2=1$.  Then $\overline{g}=\frac{1}{\varphi^{2}}g$ is a nontrivial $m$-quasi-Einstein metric with potential function $f=-m \log u$ and $\lambda\leq 0$ if, and only if, $\varphi$ and $u$ satisfy

\begin{eqnarray}
& &	(n-2)\frac{\varphi''}{\varphi}-\frac{m}{u}\Big(u''+2\frac{\varphi'}{\varphi}u'
	\Big)=0,\label{eqthmA1}\\	
& &	 \frac{\varphi''}{\varphi}-(n-1)\frac{(\varphi')^{2}}{\varphi^{2}}+m\frac{\varphi'}{\varphi}\frac{u'}{u}=\frac{\lambda}{\varphi^{2}}. \label{eqthmA2} 
\end{eqnarray} 
\end{proposition} 

\begin{proof} Let $\varphi(\xi)$ and $u(\xi)$ be functions 
depending on $\xi,$ where $\xi=\sum_{i=1}^{n}\alpha_{i}x_{i}$,  
$\alpha_{i}\in \mathbb{R}$. We first observe that, without loss of generality, 
we may assume that $\sum_{j=1}^n\alpha_j^2=1$. In fact, otherwise, 
we can consider 
$\bar\xi= \frac{1}{\sqrt{\sum_{j=1}^n\alpha_j^2}}\sum_{i=1}^n\alpha_ix_i.$  At the same time, we have 

\begin{eqnarray}
\label{eqw}
\left\{\begin {array}{cc}
	\varphi_{x_{i}}=\varphi'\alpha_{i},\hspace{1,0cm}&\ \varphi_{x_{i}x_{j}}=\varphi''\alpha_{i}\alpha_{j}, \\
	 u_{x_{i}}=u'\alpha_{i}, \hspace{1,00cm}&\ u_{x_{i}x_{j}}=u''\alpha_{i}\alpha_{j}. 
	 \end {array}\right.
\end{eqnarray} 
Since $\overline{g}=\frac{1}{\varphi^{2}}g,$ where $g$ is the Euclidean metric, the first equation of Proposition \ref{thm1}  yields, for $i\neq j$,

\begin{equation*}
(n-2)\frac{\varphi^{''}}{\varphi}\alpha_{i}\alpha_{j}-\frac{m}{u}\big(u^{''}\alpha_{i}\alpha_{j}+2\frac{\varphi^{'}}{\varphi}u^{'}\alpha_{i}\alpha_{j}\big)=0,
\end{equation*}  which can be rewritten as

\begin{equation*}
\alpha_{i}\alpha_{j}\Big[(n-2)\frac{\varphi^{''}}{\varphi}-\frac{m}{u}\big(u^{''}+2\frac{\varphi^{'}}{\varphi}u^{'}\big)\Big]=0, \qquad i\neq j. 
\end{equation*} In order to proceed we divide the proof in two cases.

\vspace{.2in}

${\bf a)}$ If there exists a pair $(i,j)$, $i\neq j$, such that $\alpha_{i}\alpha_{j}\neq 0,$ then 
we obtain

\begin{equation}
\label{eqkey1}
(n-2)\frac{\varphi^{''}}{\varphi}-\frac{m}{u}\big(u^{''}+2\frac{\varphi^{'}}{\varphi}u^{'}\big)=0.
\end{equation} 
Then, it follows from (\ref{eqw}) and the second equation of Proposition \ref{thm1} that  
\begin{eqnarray}
\frac{\lambda}{\varphi^{2}}&=&\alpha_{i}^{2}\Big((n-2)\frac{\varphi^{''}}{\varphi}-\frac{m}{u}u^{''}-2m\frac{\varphi^{'}u^{'}}{\varphi u}\Big)\nonumber\\&&+\frac{\varphi^{''}}{\varphi}-(n-1)\frac{(\varphi^{'})^{2}}{\varphi^{2}}+m\frac{\varphi^{'}u^{'}}{\varphi u}, 
\end{eqnarray} 
where we used the fact that $\sum_{k}\alpha_{k}^{2}=1$.
In particular, using the relation between $\varphi^{''}$ and $u^{''}$ obtained in (\ref{eqkey1}) we conclude that 

\begin{equation}
\label{eqkey2}
\frac{\varphi^{''}}{\varphi}-(n-1)\frac{(\varphi^{'})^{2}}{\varphi^{2}}+m\frac{\varphi^{'}u^{'}}{\varphi u}=\frac{\lambda}{\varphi^{2}}.
\end{equation} 
Hence, \eqref{eqkey1} and \eqref{eqkey2} show that $\varphi$ and $u$ must satisfy \eqref{eqthmA1} and \eqref{eqthmA2}. 

\vspace{.2in}

${\bf b)}$ On the other hand, if for all $i\neq j$ we have $\alpha_i\alpha_j=0,$ then $\xi$ is a multiple of one variable and  without loss of generality, we may consider  $\xi=x_n.$ In this case, we get 
\[
\varphi_{x_i}=\varphi^{'}\delta_{in},\qquad u_{x_i}=u^{'}\delta_{in},\qquad 
\varphi_{x_ix_j}=\varphi^{''}\delta_{in}\delta_{jn},\qquad 
u_{x_ix_j}=u^{''}\delta_{in}\delta_{jn}. 
\] 
Therefore, the first equation of Proposition  \ref{thm1}  is trivially satisfied. However,  
the second equation reduces to two equations obtained by taking $i\neq n$ and $i=n,$ respectively. Namely,  
\begin{eqnarray}
\label{firsteq} \frac{\varphi''}{\varphi}-(n-1)\frac{(\varphi')^2}{\varphi^2}+m\frac{\varphi'}{\varphi}\frac{u'}{u}=\frac{\lambda}{\varphi^2}
\end{eqnarray} and
\begin{eqnarray}
\label{secondeq} 
(n-1)\frac{\varphi''}{\varphi}-(n-1)\frac{(\varphi')^2}{\varphi^2}-m\frac{u''}{u}-m\frac{\varphi'}{\varphi}\frac{u'}{u}=\frac{\lambda}{\varphi^2}.
\end{eqnarray}
Subtracting \eqref{firsteq} from \eqref{secondeq}, we obtain  \eqref{eqthmA1}.
Moreover, subtracting \eqref{eqthmA1} from \eqref{secondeq}, we  get \eqref{eqthmA2}.    

Consequently, we conclude that in both cases, i.e., $a)$ and $b),$ the functions $\varphi$ and $u$ satisfy both equations \eqref{eqthmA1} and \eqref{eqthmA2}.

The converse of Proposition \ref{lemA} is a straightforward computation. So, we omit the details, leaving them to the interested reader.

\end{proof}

\section{Proof of the Main Results}
\label{sec3}

Before proving our main results, notice that when we consider a metric $\bar{g}$ on $\mathbb{R}^n$ homothetic to the Euclidean metric, then for any $m\neq 0,$ the nontrivial $m$-quasi Einstein metrics, whose potential function $u$ depends on $\xi,$  can only occur when $\lambda=0$ and the function $u$ is linear in $\xi$ as it was mentioned in Remark \ref{rm1}. This fact follows immediately from Proposition \ref{lemA}. In fact, if $\varphi\neq 0$ is constant, then  for any $m\neq 0$, Eq. \eqref{eqthmA1} is equivalent to saying that $u$ is linear in $\xi$ and \eqref{eqthmA2} is equivalent to $\lambda=0.$     

 \vspace{.1in} 

We are now ready to prove Theorems \ref{thmnovo} and \ref{thmB1},  where we are assuming that $\varphi$ is not constant.

\subsection{Proof of Theorem \ref{thmnovo}}
\begin{proof}
To begin with, since $\frac{\varphi''}{\varphi}= \left(\frac{\varphi'}{\varphi}\right)'+
\left(\frac{\varphi'}{\varphi}\right)^2,$ and similarly,
$\frac{u''}{u}= \left(\frac{u'}{u}\right)'+\left(\frac{u'}{u}\right)^2,$ we can rewrite the equations of Proposition \ref{lemA} as 
\begin{eqnarray}
\label{eq1anovo}
\left\{\begin {array}{c}
	(n-2)\left(\frac{\varphi'}{\varphi}\right)'
+(m+n-2)\left(\frac{\varphi'}{\varphi}\right)^2-m \left(\frac{u'}{u}\right)'	
-m\left(\frac{u'}{u}+\frac{\varphi'}{\varphi}\right)^2=0, 
\vspace*{.1in}\\ 
	 \frac{\varphi''}{\varphi}-(n-1)\frac{(\varphi')^{2}}{\varphi^{2}}+m\frac{\varphi'}{\varphi}\frac{u'}{u}=0. 
	 \end {array}\right.
\end{eqnarray} Taking into account that $m+n-2=0$ and $n-2\neq 0,$ the first equation reduces to 
\[
\left[ (\ln\varphi)'+(\ln u)'\right] '=-[ (\ln\varphi)'+(\ln u)']^2.
\]
Thereby, $\left(\frac{1}{(\ln \varphi)'+(\ln u)' }\right)'=1$ and hence, we obtain 
\[
(\ln\varphi)'+(\ln u)'=\frac{1}{\xi+C_1} \qquad \xi+C_1\neq 0. 
\]
From this relation, it follows that 
\begin{equation}
\label{equ7}
\varphi u=C_0(\xi+C_1), 
\end{equation} where $C_1$ and $C_0>0$ are constants.

In order to proceed, we substitute
\begin{equation}\label{varphiu}
\frac{\varphi'}{\varphi}=-\frac{u'}{u}+\frac{1}{\xi+C_1}, 
\end{equation} into the second equation of \eqref{eq1anovo}. Since we are assuming that $m+n-2=0,$  we infer 
 \begin{equation}\label{eqGp}  
 G'-(n-2)G^2-(n-2)G\frac{u'}{u}=0,  
\end{equation} where 
\begin{equation}\label{eqG}
G=-\frac{u'}{u}+\frac{1}{\xi+C_1}.
\end{equation}
We point out that $G\neq 0$ on an open set. In fact, otherwise \eqref{eqG} would imply  that $u$ is 
a multiple of $\xi+C_1$  and \eqref{equ7} would imply that $\varphi$ is constant, which contradicts the hypothesis of Theorem \ref{thmnovo}.

Replacing $u'/u$ in terms of $G$ given by \eqref{varphiu} into 
\eqref{eqGp}, we have
\[
\frac{G'}{G}-\frac{n-2}{\xi+C_1}=0.
\]
Upon integrating this expression we get  
\[
 \frac{G}{(\xi+C_1)^{n-2}}=e^c,\qquad\hbox{where}\,\, c\in \mathbb{R}.  
\]
Next, since $G$ is given in terms of $u$ by \eqref{eqG},  and $u>0$,  a new integration yields
\[
\log \frac{u}{|\xi+C_1|}=-\frac{e^c}{n-2}(\xi+C_1)^{n-1}+\tilde{c},
\]
consequently,
\[
\frac{u}{|\xi+C_1|}=C_2\,e^{-C_3(\xi+C_1)^{n-1}},
\]
where $C_2=e^{\tilde{c}}$ and $C_3=e^c/(n-2)$.  Therefore,  we deduce
\[
u=C_2|\xi+C_1|e^{-C_3(\xi+C_1)^{n-1}}, \,\,\,\, \hbox{with} \,\,\, C_2>0 \,\,\,  \hbox{and} \,\,\, C_3>0.
\]
Now, we can obtain $\varphi$ from this expression  for $u$ and \eqref{equ7}, which gives 
\[
\varphi=\mbox{sgn}(\xi+C_1) C_4\, e^{C_3(\xi+C_1)^{n-1}}, \quad \mbox{with } C_4>0.
\] 
 Moreover, taking into account that $f=-m\log u$ and $m=-(n-2),$ we immediately obtain
\[
f=(n-2)\Big[\log \big(C_2|\xi+C_1|\big)-C_3\big(\xi+C_1\big)^{n-1}\Big].
\]

Conversely, a straightforward computation shows that  $\varphi$ and $u$ satisfy 
\eqref{eqthmA1} and \eqref{eqthmA2}, when $\lambda=0$ and $m+n-2=0.$ This completes the proof of the theorem.

\end{proof}

\subsection{Proof of Theorem \ref{thmB1}} 

\begin{proof} It follows from Proposition  \ref{lemA} that   $\varphi$ and $u$ must satisfy \eqref{eqthmA1} and \eqref{eqthmA2}. Since $\lambda=0,$ it follows from  \eqref{eqthmA2} that 
\[
\frac{\varphi''}{\varphi}-(n-1)\left(\frac{\varphi'}{\varphi}\right)^2+m\frac{\varphi'}{\varphi}\frac{u'}{u}=0. 
\]
Moreover, we are assuming $\varphi'\neq 0$ and hence, multiplying this equation by $\varphi/\varphi'$ we get 
\[
\frac{\varphi''}{\varphi'}-(n-1)\frac{\varphi'}{\varphi}+ m\frac{u'}{u}=0, 
\]
whose integration yields  
\[
u=C \Big(\frac{\varphi^{n-1}}{\varphi'}\Big)^{\frac{1}{m}}, 
\] where $C$ is a positive constant, and it proves \eqref{uvarphi}.

In order to proceed, we substitute this function $u$ and its derivatives into \eqref{eqthmA1} to conclude that $\varphi$ must satisfy the following differential equation
\begin{equation}
\label{eqdifvarphi}
P\frac{\varphi''}{\varphi}-Q\left(\frac{\varphi'}{\varphi}\right)^2-R\left(\frac{\varphi''}{\varphi'}\right)^2 +\frac{\varphi'''}{\varphi'}=0,
\end{equation} where $P,\, Q$ and $R$ are the following constants
\begin{equation}\label{PQR}
P=2m-1+\frac{2(n-1)}{m}, \quad  Q=\frac{(n-1)^2}{m}+n-1 \quad\hbox{and} \quad R=1+\frac{1}{m}. 
\end{equation}

Now, we introduce the function
\begin{equation}
\label{wvarphi}
w(\varphi(s))=\left(\frac{d\varphi}{ds}\right)^2.
\end{equation} In particular, we have 
\[
\varphi''(s)=\frac{1}{2}\frac{d w}{d\varphi}\quad\hbox{and}\quad 
\frac{\varphi'''}{\varphi'}=\frac{1}{2} \frac{d^2w}{d^2\varphi}.
\]
Therefore, \eqref{eqdifvarphi} guarantees that $w(\varphi)$ must satisfy the following differential equation 
\[
\frac{P}{2\varphi}\frac{dw}{d\varphi}-\frac{Q}{\varphi^2}w(\varphi)-
\frac{R}{4w}\left(\frac{dw}{d\varphi} \right)^2+\frac{1}{2} \frac{d^2w}{d^2\varphi}=0,
\]
which can be rewritten as 
\begin{equation}
\label{eqdifw}
2\varphi^2 w\frac{d^2w}{d^2\varphi}-R\varphi^2\left( \frac{dw}{d\varphi}\right)^2
+2P\varphi w\frac{dw}{d\varphi}-4Q w^2=0.
\end{equation}

Proceeding, we also introduce the function $v$ as follows
\begin{equation}
\label{vvarphi}
v(\varphi)=\frac{1}{w}\frac{dw}{d\varphi}.
\end{equation} Whence, \eqref{eqdifw} reduces to the Riccati equation
\begin{equation}
\label{eqdifv}
\frac{dv}{d\varphi}+\Big(1-\frac{R}{2}\Big)v^2+\frac{P}{\varphi}v-\frac{2Q}{\varphi^2}=0.
\end{equation}

From now on, we divide the proof in the following cases:

\begin{enumerate}
\item[{\em i})]  $m=1;$
\item[{\em ii})] $m>1.$
\end{enumerate} Then we will deal with each case separately. 

\vspace{0.40cm}

To begin with, notice  that if $m=1,$ then $R=2,$ $P=2n-1,$ $Q=n(n-1)$ and \eqref{eqdifv} reduces to 
\[
\frac{dv}{d\varphi}+\frac{(2n-1)}{\varphi}v -\frac{2n(n-1)}{\varphi^2}=0, 
\]
whose solution is given by
\[
v(\varphi)=\frac{n}{\varphi}+\frac{C_0}{\varphi^{2n-1}}, 
\] where $C_{0}$ is a constant. Moreover, since $v(\varphi)$ was defined by \eqref{vvarphi}, upon integrating we get
\[
w(\varphi)=\widetilde{C}_2\varphi^n \exp\Big({-\frac{C_1}{\varphi^{2(n-1)} }}\Big),
\] where $\widetilde{C}_2$ is a positive constant and $C_1=\frac{C_0}{2(n-1)}$. Next, taking into account that $w(\varphi)$ was defined by \eqref{wvarphi}, upon integrating we obtain $\varphi$ implicitly given by 
\[
\int{\frac{exp\left(\frac{C_1}{2\varphi^{2(n-1)} }\right)}{\varphi^{\frac{n}{2}}} }d\varphi= C_2\,\xi+C_3,
\]  where $C_2=\sqrt{\widetilde{C}_2}\neq 0$ and $C_{3}$ are 
constants, which proves \eqref{intvarphi2} and this concludes the proof of the first item of Theorem \ref{thmB1}.

\vspace{.1in}

Before proceeding, notice that if $m\neq 1$, then $R\neq 2.$ In this situation, we first consider special solutions 
$v(\varphi)$ for \eqref{eqdifv} of the form
\[
v(\varphi)=\frac{a}{\varphi}, \qquad\hbox{where} \quad a\in\mathbb{R}. 
\]
For such a solution, \eqref{eqdifv} reduces to
\begin{equation}\label{eqalg}
\Big(1-\frac{R}{2}\Big) a^2+\Big( P-1\Big)a-2Q=0.
\end{equation}
We now define 
\begin{equation}\label{gammab}
 b=(P-1)^2+4Q(2-R).
\end{equation} Hence, it follows from  (\ref{PQR}) that 
\begin{eqnarray}
b&=& \left[ 2(m-1)+2\frac{(n-1)}{m}\right]^2+4\left[ \frac{(n-1)^2}{m}+n-1\right]\frac{(m-1)}{m}\nonumber\\
 &=&  4 \left[ (m-1)^2 +\frac{3(m-1)(n-1)}{m}+\frac{(n-1)^2}{m}  \right]\nonumber\\
 &=& 4\left[(m-1)^2+\frac{n-1}{m}(3m+n-4) \label{b} \right].
 \end{eqnarray}

From now on we assume that $m>1.$ In this case, we immediately obtain $3m+n-4>0$ and in particular, we have $b>0.$ Therefore, by solving \eqref{eqalg}, we have  two particular solutions for the Riccati equation \eqref{eqdifv} given by 
\[
v_j=\frac{a_j}{\varphi}, \qquad\hbox{for}  \quad j=1,2, 
\]
 where  
\begin{equation}\label{a1a2}
a_1=\frac{1}{2-R} \left[-(P-1)+ \sqrt{b }\right] \quad \hbox{and}\quad
a_2=\frac{1}{2-R} \left[-(P-1)- \sqrt{b}\right].
\end{equation}

Proceeding, consider the function  
\[
Z(\varphi)=\exp\left(\int \left(\frac{2-R}{2}\right) (v_1-v_2)\,d\varphi\right).
\]
Taking into account that $$v_1-v_2=\frac{2\sqrt{b}}{((2-R)\varphi)},$$ it follows that 
$
Z(\varphi)=\varphi^{\sqrt{b}}$. This implies that the general solution of \eqref{eqdifv} is given by
\begin{equation}\label{generalv}
v(\varphi)=\frac{a_2-C_1\varphi^{\sqrt{b}}a_1}{\varphi (1-C_1\varphi^{\sqrt{b}})}, 
\end{equation} where $C_1\in\mathbb{R}\setminus\{0\}$,  and $a_1,\,a_2$ and $b$ are the constants given by \eqref{a1a2} and \eqref{gammab}.
Thereby, since $v(\varphi)$ was defined by \eqref{vvarphi}, integrating we obtain 
\[
w(\varphi)=C_2\left( C_1\varphi^{\sqrt{b}}-1\right)^{\frac{2}{2-R}}\varphi^{a_2}
\qquad\quad  C_1\neq 0, \; C_2>0.
\]  

In order to determine $\varphi(\xi)$ and $u(\xi),$ we use \eqref{wvarphi} to infer
\begin{equation*}
\int{ \frac{d\varphi}{\left( C_1\varphi^{\sqrt{b}}-1\right)^{\frac{1}{2-R}}\varphi^{\frac{a_2}{2}}}}=C_2\xi+C_3. 
\end{equation*} 
Finally, it suffices to use (\ref{PQR}) to arrive at 
\[
\int{ \frac{d\varphi}{\left( C_1\varphi^{\sqrt{b}}-1\right)^{\frac{m}{m-1}}\varphi^{\frac{a}{2}}}}=C_2\xi+C_3, 
\] where, for simplicity, $a=a_{2}$.  Recall that \eqref{uvarphi} 
determines $u$ in terms of $\varphi$. This concludes the proof of theorem for the second case. 

Conversely, by using  (\ref{intvarphi1}) for $m=1$ and (\ref{intvarphi2}) for $m>1,$ a straightforward computation shows that $\varphi$ and $u$ satisfy (\ref{eqthmA1}) and \eqref{eqthmA2} for $\lambda=0.$  So, the proof is completed.
 \end{proof}

\begin{acknowledgement}
E. Ribeiro Jr would like to thank the Department of Mathematics - Universidade de Bras\'ilia, where part of this work was carried out, for the warm hospitality. 
\end{acknowledgement}

\end{document}